\documentclass[12pt,twoside]{amsart}
\usepackage{}
\usepackage{amsmath, amsthm, amscd, amsfonts, amssymb, graphicx, color}
\usepackage[bookmarksnumbered, plainpages]{hyperref}

\newtheorem{thm}{Theorem}[section]
\newtheorem{cor}[thm]{Corollary}
\newtheorem{lem}[thm]{Lemma}

\theoremstyle{remark}
\newtheorem{remark}{Remark}[section]
\theoremstyle{remark}

\numberwithin{equation}{section}

\def\p{$\mathbf{P}^{\frac{n+p}{2}}(\mathbf{C})$}
\def\3{$\mathbf{P}^{\frac{3+p}{2}}(\mathbf{C})$}
\def\m{$M^n$}
\def\lw{\overline\nabla}
\def\l{\nabla}
\def\tr{{\rm trace}}

\begin{document}

\title{\bf On CR Submanifolds of Maximal CR Dimension with Flat Normal Connection of a Complex Projective Space}
\author{Liang Zhang$^\ast$, Man Su, Pan Zhang}

\address{School of Mathematics and Computer Science, Anhui Normal University\\
Anhui 241000, P.R. China
}
\email{zhliang43@163.com}

\address{School of Mathematics and Computer Science, Anhui Normal University\\
Anhui 241000, P.R. China
}
\email{1181293897@qq.com}

\address{School of Mathematical Sciences\\
University of Science and Technology of China\\
Anhui 230026, P.R. China\\
}

\email{panzhang@mail.ustc.edu.cn}

\thanks{{\scriptsize
\hskip -0.4 true cm
\newline \hskip -0.4 true cm \textit{2010 Mathematics Subject Classification.} 53C15; 53C40; 53B20
\newline $^\ast$Corresponding author.
}}

\maketitle

\begin{abstract}
 In this paper, we study the CR submanifolds of maximal CR dimension with flat normal connection of a complex projective space. We first investigate the position of the umbilical normal vector in the normal bundle, especially for the submanifolds of dimension 3. Then as the application, we prove the non-existence of a class of CR submanifolds of maximal CR dimension with flat normal connection.
\end{abstract}

\vskip 0.2 true cm


\pagestyle{myheadings}
\markboth{\rightline {\scriptsize L. ZHANG, M. SU AND P. ZHANG}}
         {\leftline{\scriptsize On CR Submanifolds of Maximal CR Dimension with Flat Normal Connection }}

\bigskip
\bigskip


\section{\bf Introduction}
\vskip 0.4 true cm

Let $\mathbf{P}^{\frac{n+p}{2}}(\mathbf{C})$ be a $\frac{n+p}{2}$-dimensional complex projective space with constant holomorphic sectional curvature 4. Let $M^n$ be an $n$-dimensional real submanifold of $\mathbf{P}^{\frac{n+p}{2}}(\mathbf{C})$  and $J$ be the complex structure of $\mathbf{P}^{\frac{n+p}{2}}(\mathbf{C})$. For any point $x\in M^n$,
$$JT_x(M^n)\cap T_x(M^n)$$
is the maximal $J$-invariant subspace of the tangent space $T_x(M^n)$. We call it {\it the holomorphic tangent space} and denote it by $H_x(M^n)$. We also call the orthogonal complement of $H_x(M^n)$ {\it the totally real part of} $T_x(M^n)$ and denote it by $R_x(M^n)$. In general the dimension of $H_x(M^n)$ varies depending on the point $x\in M^n$. If $H_x(M^n)$ has constant dimension with respect to $x\in M^n$, then the submanifold $M^n$ is called {\it a CR submanifold} and the constant complex dimension is called {\it the CR dimension} of $M^n$. It is pointed out in \cite{MDMO1} that this notion of CR submanifolds is a generalization of the notion of CR submanifolds given by A.Bejancu in \cite{AB}. But there exists a special case in which the two notions coincide, that is $M^n$ has {\it maximal CR dimension}, i.e., the CR dimension of $M^n$ is $\frac{n-1}{2}$. In this case, there exists a unit normal vector $\xi_x$ such that
$$JT_x(M^n)\subset T_x(M^n)\oplus span\{\xi_x\}$$
for any point $x\in$\m.

A real hypersurface is a typical example of a CR submanifold of maximal CR dimension. The study of real hypersurfaces in complex space forms is a classical topic in differential geometry (see \cite{RT,MO,YM,MK,SMAR,RNPJR,BYCSM}) and the generalization of some results which are valid for real hypersurfaces to CR submanifolds of maximal CR dimension may be expected. For instance, some classification results of CR submanifolds of maximal CR dimension under some certain conditions were obtained in \cite{MDMO2,MDMO3,MDMO4,MDMO5,MD,THL}.

In this paper, we study CR submanifolds of maximal CR dimension with flat normal connection of a complex projective space. Note that the curvature tensor of the normal connection of a hypersurface vanishes automatically, so the normal connection of a hypersurface is naturally flat. We first discuss the position of the umbilical normal vector in the normal bundle and prove the following theorem.

\begin{thm}
Let \m be a CR submanifold of maximal CR dimension of \p, $n>2$. Let $\eta$ be an umbilical normal vector of $M^n$. If the normal connection is flat, then $\eta$ is in the direction of the distinguished normal vector $\xi$.
\end{thm}

It is proved in \cite{YTST} that there exists neither totally geodesic real hypersurfaces nor totally umbilical real hypersurfaces of a complex projective space. From Theorem 1.1, we can generalize this result to the following

\begin{cor}
  In \p ($n>2$) there exists neither totally geodesic CR submanifolds of maximal CR dimension nor totally umbilical CR submanifolds of maximal CR dimension, whose normal connections are flat.
\end{cor}

Next we consider the converse of Theorem 1.1. For 3-dimensional submanifolds, we prove the following theorem.

\begin{thm}
 Let $M^3$ be a 3-dimensional CR submanifold of maximal CR dimension of \3, $p>1$. If the normal connection is flat, then $p=3$ and the distinguished normal vector $\xi$ is umbilical.
\end{thm}

As the application of Theorem 1.1 and Theorem 1.3, we prove the non-existence of a class of CR submanifolds of maximal CR dimension of a complex projective space.

\begin{thm}
  In \3 $(p>1)$ there exist no 3-dimensional pseudo-umbilical CR submanifolds of maximal CR dimension with flat normal connection.
\end{thm}

\begin{cor}
  In \3 $(p>1)$ there exist no 3-dimensional minimal CR submanifolds of maximal CR dimension with flat normal connection.
\end{cor}

\begin{remark}
We should note that for some other ambient spaces there may exist pseudo-umbilical submanifolds with flat normal connection. For instance, from results of \cite{BYC1} we know that minimal surfaces of a hypersurface of a Euclidean space $\mathbf{E}^m$ and the product of two plane circles in $\mathbf{E}^4$ are both pseudo-umbilical with flat normal connection.
\end{remark}

\section{\bf Preliminaries}
\vskip 0.4 true cm

Let $M^n$ be a CR submanifold of maximal CR dimension of $\mathbf{P}^{\frac{n+p}{2}}(\mathbf{C})$. For each point $x\in M^n$, the real dimension of the holomorphic tangent space $H_x(M^n)$ is $n-1$. Therefore $M^n$ is necessarily odd-dimensional and there exists a unit normal vector $\xi_x$ such that
$$JT_x(M^n)\subset T_x(M^n)\oplus span\{\xi_x\}.$$
Write
\begin{equation}\label{defU}
   U_x=-J\xi_x.
\end{equation}
It is easy to see that $U_x$ is a unit tangent vector of $M^n$ which spans the totally real tangent space $R_x(M^n)$. So a tangent vector $Z_x$ of $M^n$ is a holomorphic tangent vector, i.e., $Z_x\in H_x(M^n)$ if and only if $Z_x$ is orthogonal to $U_x$. For any $X\in TM^n$, we may write
\begin{equation}\label{defFu}
  JX=FX+u(X)\xi,
\end{equation}
where $F$ is a skew-symmetric endomorphism acting on $TM^n$, $u$ is the one form dual to $U$. It is proved in \cite{MDMO1} that
\begin{equation}\label{F2}
 F^2X=-X+u(X)U,
\end{equation}
\begin{equation}\label{FU}
 u(FX)=0,\ FU=0,
\end{equation}
which imply $M^n$ has an almost contact structure.

Let $T^{\bot}_1(M^n)$ be the subbundle of the normal bundle $T^{\bot}(M^n)$ defined by
$$
  T^{\bot}_1(M^n)=\{\eta\in T^{\bot}(M^n)|\langle \eta,\xi\rangle=0\},
$$
where $\langle,\rangle$ is the inner product of the tangent space of \p. Since $T^{\bot}_1(M^n)$ is $J$-invariant, we can choose a local orthonormal basis of $T^{\bot}(M^n)$ in the following way:
\begin{equation}\label{nframe}
    \xi,\ \xi_1,\cdots,\xi_q,\ \xi_{1^*},\cdots,\xi_{q^*},
\end{equation}
where $\xi_{a^*}=J\xi_a,\ a=1,\cdots,q$ and $q=\frac{p-1}{2}$.

Let $A, A_a, A_{a^*}$ denote the shape operators for the normals $\xi, \xi_a, \xi_{a^*}$, respectively. Write
$$
   D\xi=\sum_a(s_a\xi_a+s_{a^*}\xi_{a^*}),
$$
$$
   D\xi_a=-s_a\xi+\sum_b(s_{ab}\xi_b+s_{ab^*}\xi_{b^*}),
$$
$$
   D\xi_{a^*}=-s_{a^*}\xi+\sum_b(s_{a^*b}\xi_b+s_{a^*b^*}\xi_{b^*}),
$$
where $s$'s are the coefficients of the normal connection $D$. Let $\lw$ be the connection of \p. By using the classical Weingarten formula and noting that $\lw J=0$, one can obtain the following relations (\cite{MDMO1}):
\begin{equation}\label{Aa*}
    A_{a^*}X=FA_aX-s_a(X)U,
\end{equation}
\begin{equation}\label{Aa}
    A_{a}X=-FA_{a^*}X+s_{a^*}(X)U,
\end{equation}
\begin{equation}\label{trAa*}
    \tr A_{a^*}=-s_a(U),\ \tr A_a=s_{a^*}(U),
\end{equation}
\begin{equation}\label{sa*}
    s_{a^*}(X)=\langle A_aU,X\rangle,
\end{equation}
\begin{equation}\label{sa}
    s_{a}(X)=-\langle A_{a^*}U,X\rangle,
\end{equation}
\begin{equation}\label{sab}
    s_{a^*b^*}=s_{ab},\ s_{a^*b}=-s_{ab^*},
\end{equation}
\begin{equation}\label{gbU}
    \l_XU=FAX,
\end{equation}
where $X,Y$ are tangent to $M^n$, $\l$ is the connection induced from $\lw$, and $a,b=1,\cdots,q$.

To prove our theorems, we need to write the classical equations of Codazzi and Ricci for submanifolds. For the sake of convenience, set $\xi_0=\xi$ and $\alpha,\beta=0,1,\cdots,q,1^*,\cdots,q^*$. Recall the equation of Codazzi for the normal vector $\xi$ is given by \cite{MDMO1}
\begin{align}\label{CodazziA}
  (\nabla_XA)Y-(\nabla_YA)X =&-(\overline{R}(X,Y)\xi)^{\top}+\sum_b\{s_b(X)A_bY-s_b(Y)A_bX\}\notag\\
  & +\sum_b\{s_{b^*}(X)A_{b^*}Y-s_{b^*}(Y)A_{b^*}X\},
\end{align}
where $\overline{R}$ is the Riemannian curvature tensor of $\mathbf{P}^{\frac{n+p}{2}}(\mathbf{C})$, $X,Y$ are tangent to $M^n$, $(\overline{R}(X,Y)\xi)^{\top}$ is the tangent part of $\overline{R}(X,Y)\xi$, and $(\l_XA)Y$ is defined as
\begin{equation}\label{deflA}
    (\l_XA)Y=\l_XAY-A(\l_XY).
\end{equation}
Recall that the equation of Ricci is given by \cite{MDMO1}
\begin{equation}\label{eqRicci}
    \langle R^{\bot}(X,Y)\xi_{\alpha},\xi_{\beta}\rangle=\langle \overline{R}(X,Y)\xi_{\alpha},\xi_{\beta}\rangle+\langle [A_{\alpha},A_{\beta}]X,Y\rangle,
\end{equation}
where $R^{\bot}$ is the curvature tensor of the normal connection, and
\begin{equation*}
    [A_{\alpha},A_{\beta}]=A_{\alpha}\circ A_{\beta}-A_{\beta}\circ A_{\alpha}.
\end{equation*}

Note that the Riemannian curvature tensor $\overline{R}$ of \p is given by
\begin{equation}\label{curvatureP}
\overline{R}(\overline{X},\overline{Y})\overline{Z}=  \langle\overline{Y},\overline{Z}\rangle\overline{X}
-\langle\overline{X},\overline{Z}\rangle\overline{Y}+\langle J\overline{Y},\overline{Z}\rangle J\overline{X}
-\langle J\overline{X},\overline{Z}\rangle J\overline{Y}+2\langle \overline{X},J\overline{Y}\rangle J\overline{Z},
\end{equation}
where $\overline{X},\overline{Y},\overline{Z}$ are tangent to \p. From (\ref{curvatureP}),(\ref{defU}),(\ref{defFu}), we calculate
\begin{equation*}
  \overline{R}(X,Y)\xi=u(Y)FX-u(X)FY+2\langle FX,Y\rangle U,
\end{equation*}
\begin{equation*}
  \overline{R}(X,Y)\xi_a=-2\langle FX,Y\rangle\xi_{a^*},
\end{equation*}
\begin{equation*}
  \overline{R}(X,Y)\xi_{a^*}=2\langle FX,Y\rangle\xi_{a}.
\end{equation*}
Therefore the equation of Codazzi (\ref{CodazziA}) becomes \cite{MDMO1}
\begin{align}\label{equationCodazziA}
    (\nabla_XA)Y-(\nabla_YA)X = & u(X)FY-u(Y)FX-2\langle FX,Y\rangle U \notag\\
    & +\sum_b\{s_b(X)A_bY-s_b(Y)A_bX\}\notag\\
   &+\sum_b\{s_{b^*}(X)A_{b^*}Y-s_{b^*}(Y)A_{b^*}X\}.
\end{align}
The equation of Ricci (\ref{eqRicci}) becomes
\begin{equation}\label{eqRicciAAa}
    \langle R^{\bot}(X,Y)\xi,\xi_{a}\rangle=\langle [A,A_{a}]X,Y\rangle,
\end{equation}
\begin{equation}\label{eqRicciAAa*}
    \langle R^{\bot}(X,Y)\xi,\xi_{a^*}\rangle=\langle [A,A_{a^*}]X,Y\rangle,
\end{equation}
\begin{equation}\label{eqRicciAaAb}
    \langle R^{\bot}(X,Y)\xi_a,\xi_b\rangle=\langle [A_a,A_b]X,Y\rangle,
\end{equation}
\begin{equation}\label{eqRicciAaAb*}
    \langle R^{\bot}(X,Y)\xi_a,\xi_{b^*}\rangle=-2\langle FX,Y\rangle\delta_{ab}+\langle [A_a,A_{b^*}]X,Y\rangle,
\end{equation}
\begin{equation}\label{eqRicciAa*Ab*}
    \langle R^{\bot}(X,Y)\xi_{a^*},\xi_{b^*}\rangle=\langle [A_{a^*},A_{b^*}]X,Y\rangle.
\end{equation}

\section{\bf The Position of the Umbilical Normal Vector in the Normal Bundle}
\vskip 0.4 true cm

Let $M^n$ be a CR submanifold of maximal CR dimension of $\mathbf{P}^{\frac{n+p}{2}}(\mathbf{C})$. The normal connection $D$ is said to be {\it flat}, if the curvature tensor $R^{\bot}$ of $D$ vanishes. In this section we discuss the position of the umbilical normal vector in the normal bundle for this kind of submanifolds. Recall that a normal vector $\eta$ is said to be {\it umbilical}, if the shape operator with respect to $\eta$ is given by
\begin{equation}\label{3AelxMn}
  A_{\eta}=\lambda id: T_x(M^n)\to T_x(M^n),
\end{equation}
where $\lambda=\langle\eta,\zeta\rangle$, $\zeta$ is the mean curvature vector, and $id:T_x(M^n)\to T_x(M^n)$ is the identity map. Specially, if $\zeta$ is umbilical, then the submanifold $M^n$ is called {\it pseudo-umbilical}. It is obvious that minimal submanifolds must be pseudo-umbilical (see \cite{BYC2}).

From equations of Ricci (\ref{eqRicciAAa})-(\ref{eqRicciAa*Ab*}), we see that flat normal connection implies that
\begin{equation}\label{3AAaAa0}
    [A,A_a]=0,\ [A,A_{a^*}]=0,
\end{equation}
\begin{equation}\label{3AaAAb0}
    [A_a,A_b]=0,\ [A_{a^*},A_{b^*}]=0,
\end{equation}
\begin{equation}\label{3AaAdab}
   [A_{a},A_{b^*}]=2\delta_{ab}F,
\end{equation}
where $a,b=1\cdots q$.

\begin{proof}[Proof of Theorem 1.1]
  The result trivially holds when $p=1$. In the following, we assume $p>1$. For the umbilical normal vector $\eta$, we decompose it as $\eta=\eta_1+\eta_2$, where $\eta_1\in span\{\xi\}, \eta_2\bot\xi$. Choose the unit normal vector $\xi_1$ such that $\eta_2=|\eta_2|\xi_1$, then
  \begin{equation*}
    \eta=|\eta_1|\xi+|\eta_2|\xi_1.
  \end{equation*}
  From the definition of the umbilicity of $\eta$ (see (\ref{3AelxMn})), we deduce that
  \begin{align*}
    0 & =[A_{\eta},A_{1^*}]=[|\eta_1|A+|\eta_2|A_1,A_{1^*}]\notag\\
    & =|\eta_1|[A,A_{1^*}]+|\eta_2|[A_1,A_{1^*}].
  \end{align*}
  Substituting (\ref{3AAaAa0}) and (\ref{3AaAdab}) into the above formula, we get
  \begin{equation*}
    2|\eta_2|F=0.
  \end{equation*}
  Since $n>2$ and $rank F=n-1$, we conclude that $|\eta_2|=0$. Therefore $\eta=|\eta|\xi$.
\end{proof}

To prove Theorem 1.3, we need the following lemmas. The first one is an easy linear algebra result which can be obtained by direct calculations.

\begin{lem}
  Let $(V,\langle,\rangle)$ be an $n$-dimensional inner product space and $f:V\to V$ be a linear transformation. Suppose there exist $\lambda\in \mathbf{R}$ and $X\in V$ such that $f(X)=\lambda X$. If the linear transformations $f_1,f_2:V\to V$ are both commutative with $f$, then we have
  \begin{equation*}
    f(f_1X)=\lambda f_1X,\ f(f_2X)=\lambda f_2X,\ f([f_1,f_2]X)=\lambda [f_1,f_2]X.
  \end{equation*}
\end{lem}


\begin{lem}
  Let $M^3$ be a 3-dimensional CR submanifold of maximal CR dimension of \3, $p>1$. If the normal connection is flat, then $p=3$.
\end{lem}

\begin{proof}
  Otherwise\  $p>3$, then we may choose orthonormal frame
  \begin{equation*}
  \xi,\ \xi_1,\ \xi_2,\cdots,\xi_q,\ \xi_{1^*},\ \xi_{2^*},\cdots,\xi_{q^*}
  \end{equation*}
  of $T^{\bot}(M^3)$. In the following we consider the eigenvalues and eigenvectors of the shape operator $A_1$. We prove first that if there exists an eigenvalue of $A_1$, say $\alpha$, such that $U$ is not the eigenvector corresponding to $\alpha$, then the multiplicity of $\alpha$ is 2. In fact, since the normal connection is flat, from (\ref{3AaAAb0}) and (\ref{3AaAdab}), we have
  \begin{equation*}
    [A_1,A_2]=0,\ [A_1,A_{2^*}]=0.
  \end{equation*}
  According to Lemma 3.1, if $X$ is an eigenvector corresponding to $\alpha$, then
  \begin{equation*}
    A_1([A_2,A_{2^*}]X)=\alpha [A_2,A_{2^*}]X.
  \end{equation*}
  Noting that $[A_2,A_{2^*}]=2F$, the above formula becomes
  \begin{equation*}
    A_1(FX)=\alpha FX.
  \end{equation*}
  It is easy to see that if $X\not\in span\{U\}$, then $X$ and $FX$ are linearly independent. Hence the above formula implies the multiplicity of $\alpha$ is at least 2. This combined with Theorem 1.1 shows that the multiplicity of $\alpha$ is 2.

  Next we prove $A_1$ has two distinct eigenvalues, and $U$ is the eigenvector corresponding to the simple one, while all the holomorphic tangent vectors are eigenvectors corresponding to the other one whose multiplicity is 2. In fact, Theorem 1.1 guarantees that $A_1$ has at least two distinct eigenvalues, say $\alpha$ and $\beta$. From the declaration above, we know that $U$ is an eigenvector corresponding to $\alpha$ or $\beta$, say $\beta$ (otherwise dim$M^3\geqq 4$). Then the eigenvectors of $\alpha$ are orthonormal to $U$. Also from the declaration above, we see that $\alpha$ has multiplicity 2.

  In entirely the same way we can prove that $A_{1^*}$ also has two distinct eigenvalues, and $U$ is an eigenvector corresponding to the simple one, while all the holomorphic tangent vectors are eigenvectors corresponding to the other one whose multiplicity is 2.

  Take a holomorphic tangent vector $X\not=0$. Assume that
  \begin{equation*}
    A_1X=\alpha X,\ A_{1^*}X=\alpha^*X.
  \end{equation*}
  By a direct calculation, we have
  \begin{equation*}
    [A_1,A_{1^*}]X=0.
  \end{equation*}
  On the other hand, (\ref{3AaAdab}) implies that $[A_1,A_{1^*}]X=2FX\not=0$. This contradiction shows that $p=3$.
\end{proof}

\begin{lem}
  Let $M^3$ be a 3-dimensional CR submanifold of maximal CR dimension of \3, $p>1$. If the normal connection is flat, then either the distinguished normal vector $\xi$ is umbilical, or the shape operator $A$ has two distinct eigenvalues. In the latter case, $U$ is an eigenvector corresponding to the simple eigenvalue, while all the holomorphic tangent vectors are eigenvectors corresponding to the eigenvalue with multiplicity 2. In this case, $U$ is also the eigenvector of $A_1$ and $A_{1^*}$.
\end{lem}

\begin{proof}
  From Lemma 3.2, we know that $p=3$. Choose orthonormal frame $\xi,\ \xi_1,\ \xi_{1^*}$ of $T^{\bot}(M^3)$. Since the normal connection is flat, from (\ref{3AAaAa0}) and (\ref{3AaAdab}), we have
  \begin{equation}\label{flatcom}
    [A,A_1]=0,\ [A,A_{1^*}]=0,\ [A_1,A_{1^*}]=2F.
  \end{equation}
  By the same discussion as in the proof of Lemma 3.2, we know that if $A$ has at least two distinct eigenvalues, then $A$ has two distinct eigenvalues and $U$ is an eigenvector corresponding to the simple eigenvalue, while all the holomorphic tangent vectors are eigenvectors corresponding to the eigenvalue with multiplicity 2. Assume that $AU=\mu U$. From Lemma 3.1 and (\ref{flatcom}), we have
  \begin{equation*}
    A(A_1U)=\mu A_1U,\ A(A_{1^*}U)=\mu A_{1^*}U.
  \end{equation*}
  Noting that $\mu$ is the simple eigenvalue of $A$ and $U$ is the corresponding eigenvector, it follows that there exist $\mu_1,\mu_{1^*}\in\mathbf{R}$, such that $A_1U=\mu_1U,\ A_{1^*}U=\mu_{1^*}U$, which imply that $U$ is also the eigenvector of $A_1$ and $A_{1^*}$.
\end{proof}

With the above three lemmas, we can prove Theorem 1.3.

\begin{proof}[Proof of Theorem 1.3]
  Lemma 3.2 shows that $p=3$. Now we prove $\xi$ is umbilical. Otherwise, from Lemma 3.3, we know that $A$ has two distinct eigenvalues, say $\lambda,\mu$. Assume $\mu$ is the simple one, then
  \begin{equation}\label{3AUmZlZ}
    AU=\mu U,\ AZ=\lambda Z,
  \end{equation}
  where $Z$ is any holomorphic tangent vector of $M^3$.

  Let $\zeta$ be the mean curvature vector, we decompose it as $\zeta=\zeta_1+\zeta_2$, where $\zeta_1\in span\{\xi\}, \zeta_2\bot\xi$. Choose the unit normal vector $\xi_1$ such that $\zeta_2=|\zeta_2|\xi_1$, then
  \begin{align*}
    \zeta = & |\zeta_1|\xi+|\zeta_2|\xi_1\\
    = & \frac{1}{3}(\tr A)\xi+\frac{1}{3}(\tr A_1)\xi_1+\frac{1}{3}(\tr A_{1^*})\xi_{1^*}.
  \end{align*}
  This implies that
  \begin{equation}\label{3traA10}
    \tr A=3|\zeta_1|,\ \tr A_1=3|\zeta_2|,\ \tr A_{1^*}=0.
  \end{equation}
  Combining (\ref{trAa*}) and (\ref{3traA10}), we see that
  \begin{equation}\label{3s1U3z2}
    s_1(U)=0,\ s_{1^*}(U)=3|\zeta_2|.
  \end{equation}
  Further, it follows from Lemma 3.3, (\ref{sa*}) and (\ref{sa}) that
  \begin{equation}\label{3A1UUU0}
    A_{1^*}U=\langle A_{1^*}U,U\rangle U=-s_1(U)U=0,
  \end{equation}
  \begin{equation}\label{3A1UZ2U}
    A_{1}U=\langle A_{1}U,U\rangle U=s_{1^*}(U)U=3|\zeta_2|U.
  \end{equation}
  Then for any $X\in T(M^3),X\bot U$, we have
  \begin{equation}\label{3s1XUX0}
    s_1(X)=-\langle A_{1^*}U,X\rangle=0,\ s_{1^*}(X)=\langle A_1U,X\rangle=0.
  \end{equation}

  Note that (\ref{3A1UUU0}) implies $0$ is an eigenvalue of $A_{1^*}$. According to Theorem 1.1, there must exist a non-zero eigenvalue of $A_{1^*}$, say $\alpha$. Then (\ref{3traA10}) shows that $-\alpha$ is also an eigenvalue of $A_{1^*}$. Assume that $X\in T(M^3), X\bot U, |X|=1$, and
  \begin{equation}\label{A1*X}
    A_{1^*}X=\alpha X.
  \end{equation}
  Write $Y=FX$, then
  \begin{equation}\label{A1*Y}
    A_{1^*}Y=-\alpha Y.
  \end{equation}
  From (\ref{Aa}),(\ref{3s1XUX0}),(\ref{A1*X}),and (\ref{A1*Y}), we have
  \begin{equation}\label{3A1XYAX}
    A_1X=-\alpha Y,\ A_1Y=-\alpha X.
  \end{equation}
  By a direct calculation, one can easily get
  \begin{equation}\label{A1A1*com}
    [A_1,A_{1^*}]X=-2\alpha^2 Y.
  \end{equation}
  On the other hand, it follows from (\ref{3AaAdab}) that
  \begin{equation}\label{A1A1*comX}
    [A_1,A_{1^*}]X=2FX=2Y.
  \end{equation}
  Comparing (\ref{A1A1*com}) and (\ref{A1A1*comX}), we get $\alpha^2=-1$. This is impossible, since the shape operator $A_{1*}$ is symmetric and its eigenvalues are all real numbers. This contradiction shows that $\xi$ is umbilical.

\end{proof}

\section{\bf None Existence of a Class of CR Submanifolds of Maximal CR Dimension of \3 with Flat Normal Connection}
\vskip 0.4 true cm

In this section we prove the non-existence of 3-dimensional pseudo-umbilical CR submanifolds of maximal CR dimension of \3 with flat normal connection. Otherwise, let $M^3$ be such a submanifold. We first study the position of the mean curvature vector $\zeta$ in the normal bundle.

\begin{lem}
  Let $M^3$ be a 3-dimensional pseudo-umbilical CR submanifolds of maximal CR dimension of \3, $p>1$. If the normal connection is flat, then the mean curvature vector $\zeta$ is in the direction of $\xi$.
\end{lem}

\begin{proof}
  We decompose $\zeta$ as $\zeta=\zeta_1+\zeta_2$, where $\zeta_1\in span\{\xi\},\ \zeta_2\bot\xi$. We need to prove $\zeta_2=0$. Otherwise, $\zeta_2\not=0$. From Theorem 1.1 we see that $\zeta_2$ is not umbilical. From Theorem 1.3 we know that $\zeta_1$ is umbilical. Hence $\zeta$ is not umbilical. This contradicts our assumption that $M^3$ is pseudo-umbilical. Therefore, $\zeta_2=0$, i.e.,$\zeta\in span\{\xi\}$.
\end{proof}

\begin{remark}
  The method we used in the proof of Lemma 4.1 is due to B.Y.Chen who, in \cite{BYCGL}, studied the umbilical normal vectors of submanifolds of a submanifold.
\end{remark}

From Theorem 1.3, $p=3$. So
\begin{equation*}
   \zeta = \frac{1}{3}(\tr A)\xi+\frac{1}{3}(\tr A_1)\xi_1+\frac{1}{3}(\tr A_{1^*})\xi_{1^*}.
\end{equation*}
Combined this with Lemma 4.1, we have
\begin{equation}\label{4traA10}
  \tr A=3|\zeta|,\ \tr A_1=0,\ \tr A_{1^*}=0.
\end{equation}
Further, it follows from (\ref{trAa*}) that
\begin{equation}\label{4s1U1U0}
    s_1(U)=0,\ s_{1^*}(U)=0.
\end{equation}

\begin{lem}
  Let $M^3$ be a 3-dimensional pseudo-umbilical CR submanifolds of maximal CR dimension of $\mathbf{P}^{\frac{3+p}{2}}(\mathbf{C})$, $p>1$. If the normal connection is flat, then $A_1U$ is a non-zero holomorphic tangent vector of $M^3$.
\end{lem}

\begin{proof}
  From (\ref{Aa}) and (\ref{4s1U1U0}),
  \begin{equation*}
    A_1U=-FA_{1^*}U.
  \end{equation*}
  Then the first formula of (\ref{FU}) implies that $A_1U$ is orthogonal to $U$, which shows that $A_1U$ is a holomorphic tangent vector. In the following, we prove $A_1U\not=0$. Otherwise, $A_1U=0$. Combining (\ref{Aa*}) and (\ref{4s1U1U0}), we also have
  \begin{equation*}
    A_{1^*}U=0.
  \end{equation*}
  Then (\ref{sa*}) and (\ref{sa}) give that
  \begin{equation}\label{4s10s10}
    s_1=0,\ s_{1^*}=0.
  \end{equation}
  From (\ref{4traA10}) and Theorem 1.1, we see that $A_{1^*}$ has non-zero eigenvalues $\alpha$ and $-\alpha$. By the same discussion as in the latter part of the proof of Theorem 1.3, one can deduce that $\alpha^2=-1$ which contradicts the fact that $\alpha$ is a real number. So $A_1U\not=0$. This completes the proof.
\end{proof}

Now write
\begin{equation}\label{4XA1YFX}
    X=A_1U,\ Y=FX.
\end{equation}
From (\ref{4s1U1U0}) and (\ref{Aa*}), it is easy to see that
\begin{equation}\label{Y}
   Y=FX=FA_1U=A_{1^*}U.
\end{equation}
Note that $\{X,Y,U\}$ are orthogonal to each other.

\begin{lem}
  With respect to the frame $\{X,Y,U\}$ chosen above, we have
  \begin{equation*}
    |X|^2=|Y|^2=1,
  \end{equation*}
  \begin{equation*}
    s_1(X)=0,\ s_1(Y)=-1,\ s_{1^*}(X)=1, s_{1^*}(Y)=0,
  \end{equation*}
  and the mean curvature $|\zeta|=constant$.
\end{lem}

\begin{proof}
  From (\ref{sa*}), (\ref{sa}), (\ref{4XA1YFX}) and (\ref{Y}), we have
  \begin{equation}\label{4s1XYX0}
   s_1(X)=-\langle A_{1^*}U,X\rangle=-\langle Y,X\rangle=0,
  \end{equation}
  \begin{equation}\label{4s1YXY0}
   s_{1^*}(Y)=\langle A_{1}U,Y\rangle=\langle X,Y\rangle=0,
  \end{equation}
  \begin{equation}\label{4s1YYY2}
   s_1(Y)=-\langle A_{1^*}U,Y\rangle=-|Y|^2,
  \end{equation}
  \begin{equation}\label{4s1XXX2}
   s_{1^*}(X)=\langle A_{1}U,X\rangle=|X|^2.
  \end{equation}
Applying (\ref{4s1U1U0}),(\ref{Y}) and (\ref{4s1XYX0}) to the equation of Codazzi (\ref{equationCodazziA}), we get
\begin{equation}\label{4XAU1XY}
    (\l_XA)U-(\l_UA)X=-Y+s_{1^*}(X)Y.
\end{equation}
On the other hand, we calculate
\begin{align}\label{4XAUUzX}
      & (\l_XA)U-(\l_UA)X \notag\\
    = & \l_X(|\zeta|U)-A(\l_XU)-\l_U(|\zeta|X)+A(\l_UX)\notag\\
    = & (X|\zeta|)U-(U|\zeta|)X.
\end{align}
In the above calculation we use the fact that $AZ=|\zeta|Z$ for any tangent vector $Z$ of $M^3$, which can be deduced from the pseudo-umbilicity of $M^3$ and Lemma 4.1. Combining (\ref{4XAU1XY}) and (\ref{4XAUUzX}), we get
\begin{equation*}
    (X|\zeta|)U-(U|\zeta|)X+(1-s_{1^*}(X))Y=0.
\end{equation*}
So
\begin{equation}\label{4Xz01X1}
   X|\zeta|=0,\ U|\zeta|=0,\ s_{1^*}(X)=1.
\end{equation}
Similarly, applying the equation of Codazzi (\ref{equationCodazziA}) to $(\l_YA)U-(\l_UA)Y$, we get
\begin{equation}\label{4Yz01Y1}
  Y|\zeta|=0,\ U|\zeta|=0,\ s_{1}(Y)=-1.
\end{equation}
From (\ref{4s1YYY2}), (\ref{4s1XXX2}), (\ref{4Xz01X1}), and (\ref{4Yz01Y1}), we know that
\begin{equation*}
    |X|^2=1,\ |Y|^2=1,\ |\zeta|=constant.
\end{equation*}
\end{proof}


\begin{lem}
  For the holomorphic tangent vector $X,Y$ defined by (\ref{4XA1YFX}) and (\ref{Y}), we have
  \begin{equation*}
    A_1X=U,\ A_1Y=0,\ A_{1^*}X=0,\ A_{1^*}Y=U.
  \end{equation*}
  \end{lem}

\begin{proof}
  From (\ref{Aa*}), (\ref{Aa}) and Lemma 4.3, we have
  \begin{equation}\label{4A1X1XU}
    A_1X=-FA_{1^*}X+U,
  \end{equation}
  \begin{equation}\label{4A1YA1Y}
    A_1Y=-FA_{1^*}Y,
  \end{equation}
  \begin{equation}\label{4A1XA1X}
    A_{1^*}X=FA_{1}X,
  \end{equation}
  \begin{equation}\label{4A1Y1YU}
    A_{1^*}Y=FA_{1}Y+U.
  \end{equation}
  From (\ref{3AaAdab}), we get
  \begin{equation*}
    [A_1,A_{1^*}]U=0.
  \end{equation*}
  Substituting (\ref{4XA1YFX}) and (\ref{Y}) into the above formula, we have
  \begin{equation}\label{4A1YA1X}
    A_1Y=A_{1^*}X.
  \end{equation}
  From (\ref{4A1X1XU}), (\ref{4A1YA1X}) and the skew-symmetry of $F$, we calculate
  \begin{equation}\label{4A1X1YY}
    \langle A_1X,X\rangle=-\langle FA_{1^*}X,X\rangle=\langle A_{1^*}X,Y\rangle=\langle A_1Y,Y\rangle.
  \end{equation}
  From (\ref{4traA10}), (\ref{4s1U1U0}) and (\ref{sa*}), we calculate
  \begin{align}\label{40tr1YY1}
    0 = & \tr A_1=\langle A_1X,X\rangle+\langle A_1Y,Y\rangle+\langle A_1U,U\rangle\notag\\
    = & \langle A_1X,X\rangle+\langle A_1Y,Y\rangle+s_{1^*}(U)\notag\\
    = & \langle A_1X,X\rangle+\langle A_1Y,Y\rangle.
  \end{align}
  Combining (\ref{4A1X1YY}) and (\ref{40tr1YY1}), we know that
  \begin{equation}\label{4A1XYY0}
    \langle A_1X,X\rangle=0,\  \langle A_1Y,Y\rangle=0.
  \end{equation}
  From (\ref{4A1XA1X}) and the skew-symmetry of $F$, we calculate
  \begin{equation}\label{4A1X1XY}
    \langle A_{1^*}X,X\rangle=\langle FA_1X,X\rangle=-\langle A_1X,Y\rangle.
  \end{equation}
  On the other hand, from (\ref{4A1YA1X}),
  \begin{equation}\label{4A1X1XY2}
    \langle A_{1^*}X,X\rangle=\langle A_1Y,X\rangle=\langle A_1X,Y\rangle.
  \end{equation}
  Combining (\ref{4A1X1XY}) and (\ref{4A1X1XY2}), we have
  \begin{equation}\label{4A1XXY0}
    \langle A_{1^*}X,X\rangle=0,\  \langle A_{1}X,Y\rangle=0.
  \end{equation}
  From (\ref{4traA10}), (\ref{4s1U1U0}), (\ref{4A1XXY0}) and (\ref{sa}), we calculate
   \begin{align}\label{40tr1YY}
    0 = & \tr A_{1^*}=\langle A_{1^*}X,X\rangle+\langle A_{1^*}Y,Y\rangle+\langle A_{1^*}U,U\rangle\notag\\
    = & \langle A_{1^*}X,X\rangle+\langle A_{1^*}Y,Y\rangle-s_{1}(U)\notag\\
    = & \langle A_{1^*}Y,Y\rangle.
  \end{align}
  From (\ref{4A1YA1X}) and (\ref{4A1XYY0}), we have
  \begin{equation}\label{4A1YYY0}
    \langle A_{1^*}Y,X\rangle=\langle A_{1^*}X,Y\rangle=\langle A_1Y,Y\rangle=0.
  \end{equation}
  Noting that $\{X,Y,U\}$ are orthonormal, by using (\ref{4A1XYY0}), (\ref{4A1XXY0}) and Lemma 4.3, we get
  \begin{align*}
    A_1X = & \langle A_{1}X,X\rangle X+\langle A_{1}X,Y\rangle Y+\langle A_{1}X,U\rangle U\notag\\
    = & s_{1^*}(X)U=U.
  \end{align*}
  Similarly, by using (\ref{sa*}), (\ref{sa}), (\ref{4A1XYY0}), (\ref{4A1XXY0}), (\ref{40tr1YY}), (\ref{4A1YYY0}) and Lemma 4.3, we also have
  \begin{equation*}
    A_1Y=0,\ A_{1^*}X=0,\ A_{1^*}Y=U.
  \end{equation*}
\end{proof}

Now we can prove Theorem 1.4.

\begin{proof}[Proof of Theorem 1.4]
  Since the co-dimension $p>1$, it follows from Theorem 1.3 that $p=3$. Let $M^3$ be a pseudo-umbilical CR submanifold of maximal CRdimension of \3 whose normal connection is flat. Let $X,Y$ be the holomorphic tangent vectors defined by (\ref{4XA1YFX}) and (\ref{Y}). From (\ref{3AaAdab}), we have
  \begin{equation}\label{4A1AX2Y}
    [A_1,A_{1^*}]X=2Y.
  \end{equation}
  On the other hand, it follows from Lemma 4.4 and (\ref{Y}) that
  \begin{equation}\label{4A1A1UY}
    [A_1,A_{1^*}]X=A_1A_{1^*}X-A_{1^*}A_1X=-A_{1^*}U=-Y.
  \end{equation}
  This is a contradiction which proves the non-existence of such submanifolds. This completes the proof.
\end{proof}


\vskip 0.5 true cm



\bigskip
\bigskip

\end{document}